\documentclass[10pt]{amsart}
\usepackage{amssymb,mathrsfs}
\usepackage{color,umoline}
\usepackage[dvipsnames]{xcolor}
\usepackage{graphicx}
\usepackage{enumitem}
\usepackage{relsize}

\def\wt#1{\widetilde{#1}}

\newcommand{\F}{\mathcal{F}}

\newcommand{\dist}{\operatorname{dist}}

\newcommand{\N}{\mathbb{N}}
\newcommand{\R}{\mathbb{R}}
\newcommand{\Z}{\mathbb{Z}}

\newcommand{\e}{\varepsilon}

\newcommand{\om}{\omega}

\newcommand{\ta}{\tau}

\newcommand{\x}{\xi}

\newcommand{\p}{\partial}

\newcommand{\rank}{\operatorname{rank}}
\newcommand{\supp}{\operatorname{supp}}

\newcommand{\ti}{\widetilde}

\newcommand{\Del}[1]{}

\newcommand{\bZ}{\mathbf Z}

\numberwithin{equation}{section}

\newtheorem{thm}{Theorem}[section]

\newtheorem{lem}[thm]{Lemma}
\newtheorem{prop}[thm]{Proposition}

\theoremstyle{remark}
\newtheorem{rem}{Remark}

\newcommand{\pqpair}{(\frac1p, \frac1q)} 
\usepackage[pdftex]{hyperref}

\newcommand{\hpqnorm}{\|H_n^d\|_{p,q}}

\date{\today}

\renewcommand{\sl}[1]{ \|_{L^{#1}(\mathbb S^d)}}

\begin{document}
\title[Spherical harmonic projection]{Sharp $\mathlarger{\mathlarger{{ L^p}}}$--$\mathlarger{\mathlarger{{ L^q}}}$ estimates  for 
\\ the spherical harmonic projection}

\author[Y. Kwon]{Yehyun Kwon}
\author[S. Lee]{Sanghyuk Lee}
\address{Department of Mathematical Sciences, Seoul National University, Seoul 151-747, Republic of Korea}
\email{kwonyh27@snu.ac.kr} 
\email{shklee@snu.ac.kr}


\begin{abstract}  
We consider $L^p$--$L^q$  estimates for the spherical harmonic projection operators and obtain sharp bounds  on a certain range of $p$, $q$. As an application, we provide a proof of off-diagonal Carleman estimates for the Laplacian, which extends the earlier results due to Jerison and Kenig \cite{JK}, and Stein \cite{St-append}. 
\end{abstract}
\maketitle

\section{Introduction}

\subsubsection*{Spherical harmonics and spectral projection} Let $\mathbb S^{d}$ be the $d$-dimensional unit sphere contained in $\mathbb R^{d+1}$ and  $\mathcal H_n^d$ be the space  of spherical harmonic polynomials of degree $n$ defined on $\mathbb S^d$. It is well known that  $L^2(\mathbb S^{d})=\bigoplus_{n=0}^\infty \mathcal H_n^d$, and $\mathcal H_n^d$ is of dimension $\sim n^{d-1}$ and the eigenspace of the Laplace-Beltrami operator $\Delta_{\mathbb S^d}$ with the eigenvalue $-n(n+d-1)$. See \cite{SW} for further details.   Let us denote by $H_n^d$  the projection operator to  $\mathcal H_n^d$.  Then,  for any $f\in L^2(\mathbb S^{d})$, we may write  
\[ f= \sum_{n=0}^\infty  H_n^df .\] 

The optimal $L^p$--$L^q$ bound for  $H_n^d$ in terms of $n$ has drawn  interest being  related to applications to various problems, for example, convergence of Riesz means on the sphere \cite{So}, Carleman estimates  in connection with  unique continuation problems (\cite{So-spectral, So-strong, Je, KT}). Even though  the bounds for  $H_n^d$ have a wide range of (and frequent) applications, as far as the authors are aware, it seems that the optimal $L^p$--$L^q$ bound for  $H_n^d$ has not been considered for general $p, q$.  In this paper we  attempt to obtain a complete characterization of $L^p$--$L^q$ estimates for the spherical harmonic projection for a certain range of $p,q$, and make clear the connection between these bounds and estimates for the Carleson-Sj\"olin type oscillatory integral operator. 

Let $p,q\in [1,\infty]$ and $p\le q$. We define  
\[ \|H_n^d\|_{p,q}:= \sup_{\|f\|_{L^p(\mathbb S^{d})}\le 1 }\|H_n^df\|_{ L^q(\mathbb S^{d})}.\] 
There are easy bounds which are basically consequences of  Parseval's identity.  Trivially,  $\|H_n^d\|_{2,2}=1$. This combined with the Cauchy-Schwarz inequality yields the bound $\|H_n^d\|_{2, \infty} \lesssim n^{\frac {d-1}2}$ since $\dim(\mathcal H_n^d)\sim n^{d-1}$.  Duality argument gives $\|H_n^d\|_{1, \infty}\lesssim n^{ {d-1}}$.  Moreover, it turns out that all those bounds are optimal. (For the meaning of the standard notations $\lesssim$ and $\sim$, see the end of this section.) However, for the other $p,q$,  it is no longer  trivial to obtain the optimal bound. As was observed in \cite{So},  for general $p,q$,    the problem of proving the optimal bound for $\|H_n^d\|_{p,q}$ is closely tied to  the optimal decay estimates for the Carleson-Sj\"olin type oscillatory integral operators \cite{St-beijing, Ho, Bo-osc, BG, Le1},  which are again related to the outstanding conjectures in harmonic analysis such as the Bochner-Riesz conjecture and the Fourier restriction conjecture on the sphere (\cite{Tao-br, TVV, TV, Tao-progress, Tao-bil, Le}). For  most recent developments  see  Bourgain and Guth \cite{BG} and  Guth, Hickman and Iliopoulou \cite{GHI}. These results are respectively  based on multilinear estimates due to Bennett, Carbery and Tao \cite{BCT} and the method of polynomial partitioning due to Guth \cite{G1, G2}. 

Motivated by Stanton and Weinstein \cite{Stanton}, Sogge \cite{So} obtained optimal bounds for $\|H_n^d\|_{p,q}$ with $p \le 2$, $q= 2$  for any dimensions $d\ge 2$.  Especially,  when $d=2$,  thanks to the well established $2$--dimensional  oscillatory integral estimates for Carleson-Sj\"olin type operators (\cite{CS, Ho}) he also obtained  bounds for $p,q$ in  a wider range.  See \cite{So} for details. Also, Sogge extended his result to the spectral projection operator on  compact manifold \cite{So-spectral}. There are also results concerning  more specialized bounds such as $\sup_{\|f\|_{L^p(\mathbb S^{d})}=1, f\in \mathcal H_n^d }\|H_n^df\|_{ L^q(\mathbb S^{d})}.$ See, for example, Dai, Feng and Tikhonov \cite{DFT},  De Carli and Grafakos \cite{DG}.

\begin{figure}\label{fig}
\includegraphics[width=210pt]{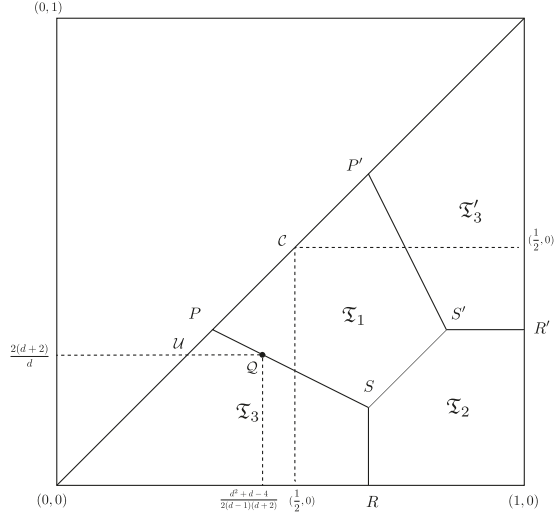}
\caption{The points $P, P', R, R'$ and $S, S'$}
\end{figure}

In order to state  our result  we need to introduce some notations. Let $I$ be the interval $[0,1]$, and define points $P=P(d)$, $R=R(d)$, $S=S(d)\in I^2$ by setting 
\[ P= \Big(\frac{d-1}{2d},\, \frac{d-1}{2d}\Big), \quad  R=\Big(\frac{d+1}{2d},\, 0\Big), \quad S=\Big(\frac{d+1}{2d},  \frac{(d-1)^2}{2d(d+1)}\Big), \]
and we also define $P'$, $R'$, $S'$ by $(x,y)'=(1-y, 1-x)$.  Let $\mathfrak T_1=\mathfrak T_1(d)$, $\mathfrak T_2=\mathfrak T_2(d)\subset I^2$ be given by 
\begin{align*}
\mathfrak T_1&=\Big\{ (x,y) \in I^2 :  y\le x,\,\,  y\ge  \frac{d-1}{d+1}(1- x),\,\,  \frac{d+1}{d-1} (1-x)\ge y,\,\,   x-y < \frac{2}{d+1}   \Big\},\\
\mathfrak T_2&=\Big\{ (x,y) \in I^2  :  x\ge \frac{d+1}{2d},\,\,     x-y \ge\frac{2}{d+1} ,\,\,  y\le \frac{d-1}{2d}  \Big\}. 
\end{align*}
We also define $\mathfrak T_3=\mathfrak T_3(d),$  and  $\mathfrak T_3'=\mathfrak T_3'(d) \subset I^2$ by setting 
\begin{align*}
\mathfrak T_3&=\Big\{ (x,y) \in I^2  :  y\le x,\,\,  y< \frac{d-1}{d+1}(1- x),\,\,  x<\frac{d+1}{2d} \Big\},\\
\mathfrak T_3'&=\Big\{ (x,y) \in I^2 :  (1-y, 1-x)\in \mathfrak T_3\Big\}. 
\end{align*} 
Note that   $\mathfrak T_1$ is the closed trapezoid with vertices  $P$, $S$, $S'$, and $P'$ from which the closed line segment $[S,S']$ is removed,    and $\mathfrak T_2$ is the closed pentagon with vertices  $R$,  $S$, $S'$, $R'$, and $(1,0)$. We also note that  the sets  $\mathfrak T_1,$ $\mathfrak T_2,$ $\mathfrak T_3,$ and $\mathfrak T_3'$ are mutually disjoint and  $(\cup_{i=1}^3  \mathfrak T_i)  \cup\, \mathfrak T_3'=\{ (x,y)\in I^2:  y\le x\}$. See Figure \ref{fig}. 

Let us set 
\[ \gamma=\gamma(p,q):=\max\Big\{  \frac{d-1}2\Big(\frac1p-\frac1q\Big),\,\,   d\Big(\frac1p-\frac1q\Big) -1,\,\,  \frac{d-1}2 -\frac dq,\,\, -\frac{d+1}2+ \frac dp\,\Big\}.\] 

\begin{thm} \label{nec}   Let $1\le p\le q\le \infty$. Then 
\begin{equation} 
\label{lower} \|H_n^d\|_{p,q}\gtrsim  n^{\gamma(p,q)}.
\end{equation} 
Futhermore, if $\pqpair\in [S,R]\cup  [S',R']$, 
\begin{equation}\label{intervalrs}
\sup_{n}\,  n^{-\gamma(p,q)}\|H_n^d\|_{p, q}  =\infty.
\end{equation}
\end{thm}

In view of Fefferman's disproof of the disk multiplier conjecture \cite{Fe}, the bound  for $(1/p, 1/q)\in [P, P']$ is not likely to be uniformly bounded except the case $p=q=2$,  and it seems possible that the lower bounds can be improved by making use of  a type of Besicovitch set. It is convenient to notice that 
\begin{equation}\label{compare}
  \gamma(p,q)= 
  \begin{cases}   
   {\frac{d-1}2(\frac1p-\frac1q)} & \text{ if }\,  (\frac1p,\frac1q)  \in \mathfrak  T_1  \,, 
   \\ 
  {d(\frac1p-\frac1q) -1} &\text{ if }\,  (\frac1p,\frac1q) \in \mathfrak  T_2    \,, 
   \\ 
  {\frac{d-1}2 -\frac dq}  &\text{ if }\,  (\frac1p,\frac1q)  \in \mathfrak  T_3   \,, 
   \\
   {-\frac{d+1}2+ \frac dp} &\text{ if }\,  (\frac1p,\frac1q)  \in \mathfrak  T_3'   \,. 
\end{cases}
\end{equation}

In the following theorem we show that, for $p\le q$, these  lower bounds  are also the upper bounds on a certain range of $p,q$, and hence prove the optimal bounds.   Combined with Theorem \ref{nec},  the following provide a complete characterization of the bounds for $\|H_n^d\|_{p,q}$.  We set 
\[  \mathcal Q=\Big(\frac{d^2+d-4}{2(d-1)(d+2)}, \frac{d}{2(d+2)}\Big), \quad  \mathcal U=\Big( \frac{d}{2(d+2)}, \frac{d}{2(d+2)}\Big),  
\quad  \mathcal C=\Big(\frac12, \frac12\Big),\] 
and define $\mathcal Q'$ and $\mathcal U'$ as before.  For given vertices $A, B, C\in I^2$ let us denote by $[A,B, C]$  the convex hulls of vertices ${A, B}, C$. 

\begin{thm}\label{k-proj}   
Let $n\ge 1$ and  $p,q$  satisfy that $1\le p \le q\le \infty$ and {$(\frac1p,\frac1q)\not\in  \big([\mathcal Q, \mathcal U, \mathcal C]\cup [\mathcal Q', \mathcal U', \mathcal C] \big)\setminus\mathcal C$ }. Then, if  $\pqpair\not\in [S,R]\cup  [S',R']$, then
\begin{equation}\label{upper} 
\|H_n^d\|_{p,q}\lesssim  n^{\gamma(p,q)}.
\end{equation} 
Futhermore, if $(\frac1p,\frac1q)=S$ or $S'$,  $\| H_n^d f \|_{L^{q,\infty}(\mathbb S^d)}\lesssim n^{1-\frac 2{d+1}} \|f\|_{L^{p,1}(\mathbb S^d)}$,  and  if $(\frac1p,\frac1q)\in (S',R']$,  $\| H_n^d f \|_{L^{q,\infty}(\mathbb S^d)}\lesssim  n^{-\frac{d+1}2+ \frac dp} \|f\|_{L^p(\mathbb S^d)}.$  Here $L^{r,s}(\mathbb S^d)$ denotes the Lorentz space. 
\end{thm}

This gives an almost complete characterization of the bounds for the 2-dimensional projection operator {$H_n^2$} when $p\le q$ except some endpoint cases. 
\begin{rem} 
The region for sharp boundedness of Theorem \ref{k-proj}, as well as that of Theorem \ref{osc-est2} and Proposition \ref{osc} below, can be further extended by making use of improved estimates in \cite{GHI}. But we do not intend to pursue it here.  Related results will appear elsewhere. 
\end{rem}

As a consequence of Theorem \ref{k-proj}  we have,  for $(\frac1 p, \frac1q)\in (S, S')$, 
\begin{equation}\label{projection}
\| H_n^d f\|_{L^q(\mathbb S^d)}\lesssim n^{1-\frac 2{d+1} } \| f\|_{L^p(\mathbb S^d)},
\end{equation}
which is crucial for the proof of the Carleman estimate  \eqref{weight}.   \eqref{intervalrs} shows the natural bound   $ \hpqnorm\lesssim  n^{d(\frac1p-\frac1q) -1}$ does not hold for $(p,q)$ with $(\frac1p,\frac1q)\in  [S,R] \cup [S', R'].$  For these $p$, $q$ we can get $  \|H_n^d\|_{p,q}\lesssim  (\log n) n^{\gamma(p,q)}$ by direct summation. This was observed in Sogge \cite{So-strong} for $(\frac1p, \frac1q)=S,$ $S'$.  However, Theorem \ref{k-proj} provides weaker substitutes without  logarithmic loss.  

For $d\ge 3$, Sogge \cite{So} obtained \eqref{upper} for $p=\frac{2(d+1)}{d+3}$, $q=\frac{2(d+1)}{d-1}$ using a $T^*T$ argument and, since $H_n^d$ is self adjoint,  interpolation between trivial estimates gives the optimal bound \eqref{upper} for $(p,q)\in [\frac{2(d+1)}{d+3},2]\times  [2,  \frac{2(d+1)}{d-1}]$. For the special case $\frac1p-\frac 1q=\frac 2d$ and $\min(|\frac12-\frac1p|, | \frac12-\frac1q|)>\frac1{2d}$,  the estimate   was obtained by Huang and Sogge \cite{HS}  with sharp bound  $n$, and recently, this was extended by Ren \cite{R} to the range $\frac 2{d+1}\le \frac1p-\frac 1q\le\frac 2d$ and $\min(|\frac12-\frac1p|, | \frac12-\frac1q|)>\frac1{2d}$, which is contained in $\mathfrak T_2$. It should be mentioned that interpolation between this estimate and the previously known bounds does not give the optimal bounds in  Theorem \ref{k-proj} when $\frac1p-\frac 1q>\frac 2d$.

\subsubsection*{Oscillatory integral estimate} $H_n^df$ is given by the zonal convolution with the zonal spherical harmonic function $\bZ_n$ of degree $n$. In fact, 
\begin{equation}\label{representation}
H_n^df (\zeta)=\int_{\mathbb S^{d}} \mathbf Z_n(\xi\cdot\zeta)  f(\xi) d\sigma(\xi).
\end{equation}
Moreover, $\bZ_n$ can be explicitly expressed by the Jacobi polynomial $P_n^{\alpha, \beta}$: 
\begin{equation} \label{zkernel}
\bZ_n (t) =C(d,n) P_n^{(\frac{d-2}2, \frac{d-2}2)}(t), \quad t\in [-1,1] , 
\end{equation} 
where $C(d,n)=\big( \frac{2n}{d-1} +1\big)\frac{\Gamma(d/2) \Gamma(d+n-1)}{\Gamma(d-1) \Gamma(n+d/2)}$. After several steps of reduction which makes use of the asymptotic expansion of the Jacobi polynomials  (see Theorem \ref{FW} below) it will be seen that the heart of matter is to obtain sharp bound for an oscillatory operator which is very similar to a simpler model operator  
\[  T_\lambda f(x)=\lambda^\frac{d-1}2  \int_{\mathbb R^d} (1+\lambda  |x-y|)^{-\frac{d-1}2} a(x,y) e^{i\lambda |x-y|} f(y) dy\] 
where $a$ is a smooth function with compact support.   This kind of oscillatory integral operator appears in the studies of the Bochner-Riesz multiplier operator (\cite{CS}). Since this operator has singularity near the diagonal, we are naturally led to consider  dyadic decomposition away from it. This will reduce the problem to obtaining sharp bounds for each operator which results from decomposition.  As is well known,  to obtain the optimal bounds it is important to exploit the decay in $L^p$--$L^q$ bound due to the oscillatory kernel. This leads us to consider the oscillatory integral operators which satisfy the Carleson-Sj\"olin condition. 

\subsubsection*{Carleson-Sj\"olin condition and oscillatory integral} Let $\lambda\ge 1$,  $\psi$ be a smooth function, and  $a$  be a smooth function with compact support, which are defined on $\mathbb R^d\times \mathbb R^{d-1}$. Then we set
\[   S_\lambda  f(x)=\int e^{i\lambda \psi(x,z)} a(x,z) f(z) dz, \,\, \quad (x,z)\in  \mathbb R^d\times \mathbb R^{d-1}. \] 
Suppose that on the support of $a$   
\begin{equation}\label{cscon1}
\rank   \partial_z  \partial_x \psi= d-1
\end{equation}
and, if $\pm v\in \mathbb S^{d-1}$ is the unique vector  such that $ \partial_z  ( v \cdot \partial_x \psi) =0$, then 
\begin{equation}\label{cscon2}
\rank  \partial_z^2 ( v \cdot \partial_x \psi)  =d-1. 
\end{equation}
The conditions \eqref{cscon1} and \eqref{cscon2} are equivalently stated as follows:  The map $z\to \nabla_x \psi(x,z)$ defines a family of smooth immersed surfaces with nonvanishing Gaussian curvature.  The following is due to Stein \cite{St-beijing} (also, see \cite[Ch.9]{St-book}). 

\begin{thm}  \label{osc-est} Suppose $\psi$ satisfies \eqref{cscon1} and \eqref{cscon2} in the support of $a$. Then, for $p$, $q$ satisfying  $q\ge \frac{2(d+1)}{d-1}$ and 
\begin{equation}\label{ca-so}  
\frac{d+1}q\le (d-1)(1-\frac1p),   
\end{equation} 
the  estimate  $\|S_\lambda f\|_{q} \lesssim  \lambda^{-\frac dq} \|f\|_p$  holds. 
\end{thm} 

When $d=2$ the estimate was shown to be true by H\"{o}rmander \cite{Ho} for the optimal range of $p$, $q$ (\eqref{ca-so} and $q>4$). In higher dimensions it was shown by Bourgain  \cite{Bo-osc} that the estimate generally fails if $q< \frac{2(d+1)}{d-1}$ whenever $d\ge3$ is odd. However, it was observed in  \cite[Theorem 1.3]{Le1} that  the range can be improved under the stronger  assumption that 
\begin{equation}
\label{elliptic}
\text{the surface $z\to \nabla_x \psi(x,z)$ has $d-1$ nonzero principal curvatures of  the same sign,}
\end{equation} 
equivalently, the matrix $\partial_z^2 ( v \cdot \partial_x \psi)$ is positve definite or negative definite. Under the assumption \eqref{elliptic} the range can be improved to $q> \frac{2(d+2)}{d}.$  When $d=3$, it is known that the range is sharp.  In  \cite[Remark 3.4]{BS} (also see Theorem 2.1 in \cite{BS}), by removing $\epsilon$-loss of the estimate due to the second author \cite{Le1} they obtained the bound  
\[ \Big\| \sum_{(\nu, \nu')\in \Xi_j} Th_\nu^j    Th_{\nu'}^j  \Big\|_{L^{q/2}} \lesssim   2^{2j\big (   \frac{d+1}q-(d-1)(1-\frac1p)\big)} \lambda^{-\frac{2d}q}\| h \|_p  \| h  \|_p \] 
for $q> \frac{2(d+2)}{d}$ and $p\ge 2$.  Here we keep using the notations from \cite{BS}.  By summation along $j$  this gives the estimate $\|S_\lambda f\|_{q} \lesssim  \lambda^{-\frac dq} \|f\|_p$ for $p$, $q$ satisfying  \eqref{ca-so} with strict inequality.  However, this can be combined with a simple summation trick (for example, see Lemma \ref{interpol} below and \cite[Lemma 2.6]{Le-circle})  to give  restricted weak type estimate for $p$, $q$ satisfying \eqref{ca-so}  and   $q> \frac{2(d+2)}{d}$. These estimates can be (real) interpolated to yield the strong type bound. Hence, we have the following. 

\begin{thm}  \label{osc-est2} Suppose $\psi$ satisfies \eqref{cscon1} 
and \eqref{elliptic}. Then, for $p$, $q$ satisfying  $q> \frac{2(d+2)}{d}$ and  \eqref{ca-so}, 
the  estimate  $\|S_\lambda f\|_{q} \lesssim  \lambda^{-\frac dq} \|f\|_p$  holds. 
\end{thm} 

\subsubsection*{Carleman estimate}  Let  $d\ge 3$ and   $1\le p\le q\le  \infty$.  For $\tau\in \mathbb R,$  we consider the estimate
\begin{equation}\label{weight} 
\| |x|^{-\tau} u\|_{{L^q(\mathbb R^d)} }\le C \| |x|^{-\tau} \Delta u\|_{{L^p(\mathbb R^d)} }, \quad u \in C_c^\infty(\mathbb R^d\setminus\{0\}). 
\end{equation} 
{Jerison-Kenig \cite{JK} showed that, for $p=\frac{2d}{d+2}$,  $q=\frac{2d}{d-2}$,   \eqref{weight}  holds with $C$ independent of $\tau$ under the condition $\dist(\tau, \mathbb Z+\frac{d-2}2) >0$. } Their proof is based on analytic interpolation between  $L^2$--$L^2$ and $L^1$--$L^\infty$ estimates for an analytic family of operators, which gives the  $L^\frac{2d}{d+2}$--$L^\frac{2d}{d-2}$ estimate  on the line of duality.   Bounds for $(p,q)$ other than $(\frac{2d}{d+2},  \frac{2d}{d-2})$ were also obtained by Stein \cite{St-append}.  He obtained \eqref{weight} for $p$, $q$ satisfying  $\frac12+\frac{d-2}{d(d-1)}<\frac 1p< \frac12+ \frac 1{d-1} $ and \eqref{scaling}. See \cite[Proof of Theorem 1]{St-append}. Compared to Theorem \ref{carleman} below, the range obtained by Stein is smaller when $d > 4$, coincides when $d = 4$, and is larger when $d = 3$. Jerison \cite{Je} later provided an alternative proof which is based on the bound $\|H_n^{d-1}\|_{\frac{2d}{d+2},\frac{2d}{d-2}}$ (see Section \ref{appcar}). 
 
Homogeneity dictates that \eqref{weight} is possible only if 
\begin{equation}\label{scaling}
\frac1p-\frac1q=\frac 2d. 
\end{equation} 
If $\tau\in (-d, d)$, it is easy to show  that  \eqref{weight}  holds with $C$ {depending on} $\tau$ for $1<p\le q<\infty$ satisfying \eqref{scaling}.  However, the uniformity 
of the  bound usually makes the  range of  admissible $(p,q)$ smaller. An easy argument using the spherical harmonics shows that  the uniform estimate \eqref{weight} can not be true for all $1<p\le  q<\infty$ satisfying \eqref{scaling} (see Remark \ref{fail} in the last section).  Similar phenomena were also observed in Kenig-Ruiz-Sogge \cite{KRS} where uniform Sobolev estimates for second order differential operators are obtained (also see \cite{JKL}).  

Following Jerison's argument and making use of \eqref{projection} we have the following. 

\begin{thm} \label{carleman}   Let $d\ge 3$. 
 Then, the estimate \eqref{weight} holds  uniformly in $\tau$  whenever  $\dist(\tau, \mathbb Z+\frac d{q})>0$ if $p$, $q$ satisfy \eqref{scaling}  and  $ \frac{2(d-1)d}{d^2+2d-4} < p < \frac{2(d-1)}{d}. $
\end{thm}

It is well known  \cite{JK} that under suitable conditions on $u$  the inequality \eqref{weight} implies strong unique continuation property for the inequality $|\Delta u(x)|\le V(x) |u(x)|$ when $V\in L^{d/2}_{loc}(\mathbb R^d)$.  As was shown by Stein \cite{St-append} real interpolation between estimates in Theorem \ref{carleman}  yields a refined estimate, for the same $p$, $q$  as above, $\| |x|^{-\tau} u\|_{L^{q,p}(\mathbb R^d)}  \le C \| |x|^{-\tau} \Delta u\|_{L^p(\mathbb R^d)} $  provided $u \in C_c^\infty(\mathbb R^d\setminus\{0\})$, which extends the strong unique continuation property to $V\in L^{d/2,\infty}_{loc}(\mathbb R^d)$ under smallness condition.

In Sogge \cite[final remark]{So-strong}  it was briefly (without proof) mentioned that the method therein can be applied to obtain \eqref{weight} for $p,q$ on a certain range.  The argument was based on  the weaker estimate $\|H^{d-1}_n\|_{p,q}\lesssim (\log n)\, n^{1-\frac 2d}$ for some $p$, $q$ satisfying \eqref{scaling}.  However, the proof there doesn't  seem easily accessible.\footnote{Numerology in \cite{So-strong} are not correct and should be corrected as those in this paper.}  So we decided to include Theorem \ref{carleman} and its proof which is based on the spectral projection estimate \eqref{projection}. In view of Stein's result  when $d=3$ and the range in Remark \ref{fail} below, the range in Theorem \ref{carleman} is unlikely to be optimal even in higher dimensions.   The problem of characterizing $p$, $q$ for which \eqref{weight} holds remains open. 

{\subsubsection*{Notations} For positive real numbers $X$ and $Y$, we use the notation  $X\lesssim Y$ (or $Y\gtrsim X$) to  say that there is a $C$ that $X\le CY$. Particularly, $\|H_n^d\|_{p,q}\lesssim n^{\gamma(p,q)}$ means that  $\|H_n^d\|_{p,q}\le C n^{\gamma(p,q)}$, where $C$ is a constant independent of $n$, although it may depend on $p$, $q$ and $d$. We also use the notation $X\sim Y$ which denotes $X\lesssim Y$ and $Y\lesssim X$.}

\subsubsection*{Acknowledgement} Y. Kwon was supported by NRF grant no. 2015R1A4A1041675 (Korea) and  S. Lee was supported by NRF grant no. 2015R1A2A2A05-000956 (Korea). Part of this work was carried out while  the second named author was  visiting  Departmemt of Mathematics Kyoto University.  He would like to thank Prof. Yoshio Tsutsumi for hospitality during his visit.

\section{Preliminaries} 

\subsubsection*{Asypmtotic expansion of the Jacobi polynomials}  For our purpose we need asymptotic expansion of  the polynomial $P_n^{(\frac{d-2}2, \frac{d-2}2)}(\cos\theta)$.  We use the following theorem from Frenzen and Wong \cite{Frenzen-Wong}.  A similar result was also obtained in Szeg\"o \cite{Sz2}. 

\begin{thm}\cite[Main Theorem]{Frenzen-Wong}  \label{FW}
Let $N=n+(\alpha+\beta+1)/2$.  For $\alpha>-1/2$, $ \alpha-\beta>-2m$,  $\alpha+\beta\ge -1$, and $\theta\in (0,\pi)$, we have 
\begin{align} \label{asymp}
P_n^{(\alpha,\beta)}(\cos\theta)&=\frac{\Gamma(n+\alpha+1) }{n !}
\Big(\sin\frac\theta2\Big)^{-\alpha}
\Big(\cos\frac\theta2\Big)^{-\beta}\Big(\frac\theta{\sin\theta}\Big)^{\frac12}\\
&\qquad\times \Big[ \,\sum_{l=0}^{m-1} A_l(\theta) N^{-\alpha-l} J_{\alpha+l}(N\theta)+\theta^\alpha O(N^{-m})\,\Big],\nonumber
\end{align}
where $A_l$ are analytic functions on $[0,\pi)$ and the $O$-term is uniform with respect to $\theta\in [0, \pi-\epsilon]$, $\epsilon$  being an arbitrary positive number. Here, $J_\nu$ denotes the Bessel function of order $\nu>-1/2$.
\end{thm}

In particular,  $A_0=1$. (See \cite[p. 994]{Frenzen-Wong}.) Recall the symmetric property of Jacobi polynomials:  $P_n^{(\alpha, \beta)}(-z)=(-1)^n P_n^{(\beta, \alpha)}(z)$. In particular, we have 
\begin{equation} \label{sym}
P_n^{(\frac{d-2}2, \frac{d-2}2)}(\cos (\pi-\theta))=(-1)^n P_n^{(\frac{d-2}2, \frac{d-2}2)}(\cos\theta).
\end{equation} 
By this symmetry we may get around the uniformity issue related to $\epsilon$, in the above theorem. 

\subsubsection*{Oscillatory intergral estimate}  In what follows we obtain oscillatory integral estimates which are needed  later. 
\begin{lem}\label{cosine}  Let $\epsilon>0$ and set $\psi_\epsilon (\theta) =\epsilon^{-1}\arccos(1-\epsilon^2\theta)$. Then 
$\psi_\epsilon\to \sqrt 2\theta^\frac12$ as $\epsilon\to 0$ in $C^N([2^{-8}, 2^{8}])$ for any $N$. 
\end{lem}
\begin{proof} Since $(\arccos \theta)'=-1/\sqrt{1-\theta^2}$ and $\arccos 1=0$, we have 
\[\arccos(1-\theta)=\int_{1-\theta}^1 \frac{d\tau}{\sqrt{1-\tau^2}} 
= \int_0^\theta \frac{d\tau}{\sqrt{2\tau-\tau^2}}=\sqrt 2\theta^\frac12+O(\theta^\frac32).\] 
Hence, $\psi_\epsilon(\theta)=\sqrt 2\theta^\frac12+ R_\epsilon(\theta)$ with $\|R_\epsilon\|_{C^N([2^{-8}, 2^{8}])}=O(\epsilon)$. This gives the convergence  in $C^N([2^{-8}, 2^{8}])$. 
\end{proof} 

\newcommand{\acos}{\arccos}
\begin{lem} \label{converge} Let us set 
\[ \Phi_\epsilon(x,y)=  \epsilon^{-1}\,  \acos \big( \epsilon^2 x\cdot y+ \sqrt{(1-\epsilon^2|x|^2)(1-\epsilon^2|y|^2)}\big).\] 
Then, for $|x-y|\in  [2^{-3},  2^{3}] $, $\Phi_\epsilon(x,y) \to |x-y|.$ Furthermore, the convergence holds in $C^N(K_1\times K_2)$ for any large $N$ as long as $K_1, K_2$ are compact subsets of $\mathbb R^d$ satisfying $2^{-3}\le |x- y|\le 2^3$ for $(x,y)\in K_1\times K_2$.
\end{lem}
\begin{proof} We write
\[ \epsilon^2 x\cdot y+ \sqrt{(1-\epsilon^2|x|^2)(1-\epsilon^2|y|^2)} = 1- \epsilon^2\Big( \frac{|x|^2+|y|^2-\epsilon^2|x|^2|y|^2}{1+\sqrt{(1-\epsilon^2|x|^2)(1-\epsilon^2|y|^2)}} -x\cdot y\Big),\] 
and set 
\[ \mu_\epsilon(x,y)= \frac{|x|^2+|y|^2-\epsilon^2|x|^2|y|^2}{1+\sqrt{(1-\epsilon^2|x|^2)(1-\epsilon^2|y|^2)}} -x\cdot y.\] 
Clearly $\mu_\epsilon(x,y)\to |x-y|^2/2$ as $\epsilon\to 0$ and $\Phi_\epsilon=\psi_\epsilon \circ \mu_\epsilon$. Hence, with sufficiently small $\epsilon>0$ which gives $ \mu_\epsilon(x,y)\in [2^{-8}, 2^8]$, Lemma  \ref{cosine}  shows the convergence. 
\end{proof} 

Let us define the oscillatory integral operator $T_\lambda^\epsilon$ by
\begin{equation}
T_\lambda^\epsilon f(x)= \int  a(x,y) e^{i\lambda \Phi_\epsilon(x,y)}   f(y) dy
\end{equation}
where $a\in C_c^{\infty}(\mathbb  R^d\times \mathbb R^d)$ and  $\supp a\subset  \{(x,y): 2^{-2}\le |x-y|\le 2^2\}$. 

Since $|x-y|\ge 2^{-2}$ by decomposing the support of $a$  we may assume $|x_i-y_i|\ge  (4d)^{-1}$  for some $i=1,\dots, d$. In particular, let us assume $|x_d - y_d |\ge  (4d)^{-1}$ on the support of $a$ and set  $\Phi_0(x,y)=|x-y|$, and $\Phi_0^{y_d}(x,z)=\sqrt{ |\bar x-z|^2+ (x_d-y_d)^2}$, where $x=(\bar x, x_d)$ and $y=(z,y_d)$. It is easy to show that  $\Phi_0^{y_d}$ satsfies \eqref{cscon1} and \eqref{elliptic} since  $|x_d-y_d|\sim (4d)^{-1}$.  Freezing the variable $y_d$ and  using  Minkowski's inequality and  {Theorem \ref{osc-est2}}, we see that 
 \[  \| T_\lambda^0 f\|_{q}\le C  \int  \Big\|\int a(x,z, y_d) e^{i\lambda  \Phi_0^{y_d}(x,z)} f(z, y_d) dz\,\Big\|_{q} dy_d \lesssim \lambda^{-\frac dq} \|f\|_p .\] 
 for $p$, $q$ satisfying \eqref{ca-so} {and $q>\frac{2(d+2)}{d}$}.  From  Lemma  \ref{converge}, we also see  that the same argument also works with $\Phi_\epsilon^{y_d}(x,z)=\Phi_\epsilon(x,z,y_d)$ as long as $\epsilon$ is small enough. Furthermore, since the bound   for $S_\lambda$ in Theorem \ref{osc-est} is stable  under smooth small perturbation of the phase, we see that  the oscillatory integral operators defined by $\Phi_\epsilon^{y_d}$ have uniform bounds. Hence,  from the above argument we get the following. 

\begin{prop}\label{osc}
For $p$, $q$ satisfying  \eqref{ca-so} and {$q>\frac{2(d+2)}{d}$}, there are constants  $\epsilon_0>0$ and $C>0$ such that 
\begin{equation} 
\|T_\lambda^\epsilon f\|_{q} \le C  \lambda^{-\frac dq} \|f\|_p
\end{equation} 
provided that  $ 0<\epsilon\le \epsilon_0 $.  Here $C$ is independent of $\epsilon\in(0,\epsilon_0]$ and $\lambda\ge1$.
\end{prop}

\section{Spectral projection operator: Proof of Theorem \ref{k-proj}} 
Fix $0<  r\le  10^{-2}\epsilon_0 $. Here $\epsilon_0$ is the one in   Proposition \ref{osc}.   In order to prove Theorem \ref{k-proj}, by rotational symmetry and finite decomposition of $\mathbb S^d$   it is sufficient to show that 
\begin{equation}\label{total}
 \|H_n^df\|_{L^q (\mathcal C(e_{d+1}, r )) } \lesssim  n^{\gamma(p,q)} \|f\|_p. 
\end{equation}
Here $\mathcal C(e,\rho)$  is the geodesic ball centered at $e\in\mathbb S^d$ with radius $\rho>0$.
 
We distinguish three cases  in which  $f$ is supported near the north pole $e_{d+1}$,  near the south pole $-e_{d+1}$,  and  away from both  the north pole $e_{d+1}$ and the south pole $-e_{d+1}$, respectively:
\begin{gather}
\supp f \subset \big\{  x\in \mathbb S^{d}:   x_{d+1} \in (1-100r^2, 1]\big\}, \label{north} \\
\supp f \subset \big\{  x\in \mathbb S^{d}:   x_{d+1} \in [-1, -1+100r^2)\big\}, \label{south} \\
\supp f  \subset \big\{  x\in \mathbb S^{d}:   x_{d+1} \in [-1+100r^2, 1-100r^2 ]\big\}. \label{awaynorthsouth}
\end{gather}
Recalling \eqref{representation} and \eqref{zkernel} we note that  $\arccos (\zeta\cdot\xi)\in (\pi-100r,  \pi] $ if $\zeta\in \mathcal C(e_{d+1},r)$ and $\xi\in \big\{  x\in \mathbb S^{d}:   x_{d+1} \in [-1, -1+100r^2)\big\}.$   Since  the asymptotic expansion \eqref{asymp} is uniform only for $\theta\in [0, \pi-\epsilon]$, we can not  make use of it directly when we handle  the second case \eqref{south}. However, by \eqref{sym} we may  again use \eqref{asymp} after reflection. In this manner the case \eqref{south} can be handled in the same way as the case \eqref{north}. Therefore  it  is sufficient to consider the first  and the third cases only. 

Hence, for the rest of this section,  we assume $\theta:=\arccos (\zeta\cdot\xi)\in [0, \pi- r\,]$. For simplicity  we  set $\nu=\frac{d-2}2$ and 
\[\mathcal A_\nu(\theta)=  \frac{\Gamma(n+\nu+1) }{n !} \Big(\sin\frac\theta2\Big)^{-\nu} \Big(\cos\frac\theta2\Big)^{-\nu}\Big(\frac\theta{\sin\theta}\Big)^{\frac12}. \] 
From \eqref{asymp} we may write 
\begin{equation}\label{asymp2}
P_n^{(\nu,\nu)}(\cos\theta)= \mathcal A_\nu (\theta)\Big[ \,\sum_{l=0}^{m-1} A_l(\theta) N^{-\nu-l} J_{\nu+l}(N\theta) \Big]+\mathcal E_m(\theta), 
\end{equation}  
where $\mathcal E_m(\theta)=O(N^{-m+\nu})$.  We fix $m$ large enough so that $ \mathcal E_m(\theta)=O(N^{-\frac{d} 2})$.  Hence the contribution of  $\mathcal E_m(\theta)$  to the convolution kernel \eqref{zkernel} is $O(1)$, and is negligible since $\gamma(p,q)\ge 0$.

\subsection{Away from the north and the  south poles}  In this case the bound is much better than what we need to show.  In fact,  assuming \eqref{awaynorthsouth} we show, for $p, q$ satisfying  \eqref{ca-so} and $q>\frac{2(d+2)}{d}$, 
\begin{equation}\label{nands}
\|H_n^df\|_{L^q (\mathcal C(e_{d+1}, r )) } \lesssim  n^{\frac{d-1}2-\frac dq} \|f\|_p. 
\end{equation}  
Since $f$ is supported  in $\big\{  x\in \mathbb S^{d}:   x_{d+1} \in [-1+100r^2, 1-100r^2 ]\big\}$,  $\xi\cdot\zeta=\cos\theta \in  [-1+85 r^2, 1-85 r^2]$ in \eqref{representation}.  Hence, we only need to consider the case $\theta\in [10 r, \pi- 10 r] $.  

From \eqref{zkernel} and \eqref{asymp2} we only  consider the contribution from the main term 
\[ \mathcal A_\nu (\theta) A_0(\theta) N^{-\nu} J_{\nu}(N\theta) \] 
while $\theta\in [10 r, \pi- 10 r] $.  The contribution from  the other terms  $ \mathcal A_\nu (\theta) A_l(\theta) N^{-\nu-l} J_{\nu+l}(N\theta)$, $1\le l\le m-1 $ can be handled similarly but these give smaller bounds  $n^{\frac{d-1}2-\frac dq-l}$. We recall the  asymptotic expansion of Bessel function  $J_{\nu}$ for large arguments 
\[ J_\nu (r) = r^{-\frac12} e^{ir}  \Big(\sum_{0\le j\le \frac{d+1}{2}} a_j r^{-j} \Big) +  r^{-\frac12} e^{-ir}  \Big(\sum_{0\le j\le\frac{d+1}2} b_j r^{-j} \Big)+O(r^{ -\frac {d+1}2}), \quad r\ge 1. \]
Inserting this in the above, we see that  the main term  is given by 
\[  \mathcal K_\pm ( \theta)=   N^{-\frac{1}2} \widetilde A_{\pm} (\theta) e^{\pm i N\theta} \] 
with smooth $\widetilde A_{\pm}$ which is supported in  $ [10 r, \pi- 10 r]$. As before, contribution from the lower order terms and $O(r^{-\frac {d+1}2})$ are less significant since these give smaller bounds.  Since $C(d,n)\sim n^\frac d2$,  combining the above with \eqref{zkernel}, for \eqref{nands} it suffices to show that 
\begin{equation}
\Big\| \int_{\mathbb S^d}  \widetilde A_\pm (\arccos \zeta\cdot \xi)  e^{\pm i N \arccos \zeta\cdot \xi} f(\xi) d\xi \Big\|_{L^q (\mathcal C(e_{d+1}, r )) }   \lesssim N^{-\frac dq}\|f\|_p, 
\end{equation}
whenever $f$ satisfies \eqref{awaynorthsouth}.

\newcommand{\lr}[1]{\|_{L^{#1}(\mathbb R^d)}}
\newcommand{\Lr}[1]{\Big\|_{L^{#1}(\mathbb R^d)}}

Let $x=(\bar x, v)$, $y=(\bar y,u) \in \mathbb R^{d-1}\times \mathbb R$.  Again by decomposition and {rotation in $\bar y$} we may assume $f$ is supported in the set 
\[ \Big \{(\bar y, \sqrt{1-|\bar y|^2-u^2},  u) :  |\bar y |\le \frac 12\sqrt{1-u^2},  \,  u\in [100 r^2-1, 1-100r^2] \Big \}. \] 
Hence, using the above parametrization and  parameterizing near the north pole with 
\[(\bar x,  v, \sqrt{1-|\bar x|^2-v^2}), \quad  |(\bar x,  v )|\le  r,  \] 
and ignoring the harmless  smooth factors resulted from  parametrization, we are reduced to showing   
\begin{equation} \label{nearequator} 
 \Big\|   \iint  a(\bar x, v,\bar y,u)  e^{i N  \arccos  \phi(\bar x, v, \bar y,u)}  h(\bar y,u)  d\bar y\, du \Lr {q} \lesssim N^{-\frac dq}  \|h\lr p, 
\end{equation} 
where  $a$ is a smooth function  supported in the set $ \{(x,y)\in\R^d\times\R^d:  |\bar y |\le \frac 12\sqrt{1-u^2}, \, u\in [100 r^2-1, 1-100r^2], \, |(\bar x, v)|\lesssim r\}$  and 
\[ \phi( x, y)=\bar x\cdot  \bar y + v \sqrt{1-|\bar y|^2-u^2}   + u\sqrt{1-|\bar x|^2-v^2}.\] 

Fixing $u\in ( 100r^2-1, 1-100 r^2)$, let us set 
\[  \phi_u(\bar x, v,\bar y) =\phi(\bar x, v, \bar y, u).\] 
By Minkowski's inequality, in order to show \eqref{nearequator}, it is sufficient to show that, for $p, q$ satisfying  \eqref{ca-so} and $q>\frac{2(d+2)}{d}$,
\begin{equation}\label{nearequator1} 
 \Big\|   \int  a(\bar x, v,\bar y,u)  e^{i N  \arccos  \phi_u(\bar x, v, \bar y)}  g(\bar y) \, d\bar y \Lr q \le CN^{-\frac dq}  \|g\|_{L^p(\mathbb R^{d-1})}
\end{equation} 
with  $C$ independent of $u$.   

To show \eqref{nearequator1}  by Theorem \ref{osc-est2}  it is sufficient to  check that the phase function $\arccos\phi_u$ satisfies the conditions \eqref{cscon1},  \eqref{cscon2} and \eqref{elliptic}.  We first notice that 
\[  \nabla_{\bar x, v} \arccos \phi_u  =  \frac{-1}{\sqrt{1-(\phi_u)^2}} \Big(\bar y -\frac{u \bar x}{\sqrt{1-|\bar x|^2-v^2}},\,  \sqrt{1-|\bar y|^2-u^2} - \frac{u v}{\sqrt{1-|\bar x|^2-v^2}}\Big).\] 

Also note that  $\sqrt{1-(\phi_u)^2}    \ge r$ since $ u\in (100 r^2-1, 1- 100 r^2)  $. Clearly, the map  $\bar y\to  \nabla_{\bar x, v} \arccos\phi_u(\bar x,  v, \bar y)$ defines an immersed surface. In order to show that it has nonvanishing gaussian curvature everywhere,  by rotational symmetry of the phase function  (rotation and horizontal rotation on $\mathbb S^d$) it is enough to check this by assuming 
\[ (\bar x,  v)=0, \,\, (\bar y, u)=(0, u), \,\, u\in ( 100r^2-1, 1-100 r^2). \] 

Note that $\nabla_{\bar y}  \phi_u=  \bar x - \frac{ v\bar y}{\sqrt{1-|\bar y|^2-u^2}} $. Hence $\nabla_{\bar y} \phi_u(0,0,0)=0$. Using this and  straightforward computation give
\[  \nabla_{\bar y}  \nabla_{\bar x, v} \arccos \phi_u (0, 0, 0)=  -\frac{1}{\sqrt{1-u^2}} \begin{pmatrix} I_{d-1} \\  0\end{pmatrix}. \]
So, the unique vector  $\mathbf v$ satisfying $\nabla_{\bar y} (\mathbf v \cdot \nabla_{\bar x, v} \arccos \phi_u )(0, 0,  0) =0$ is  $\pm e_d$.   We now consider 
\begin{align*}
   \Phi_u(\bar x, v, \bar y)&:=e_d\cdot  \nabla_{\bar x, v}  \arccos \phi_u \\ &\,=  \frac{-1}{\sqrt{1-(\phi_u)^2}}  \times\Big(   \sqrt{1-|\bar  y|^2-u^2} - \frac{u v}{\sqrt{1-|\bar  x|^2-v^2}}  \Big) \\&=:A\times B.
\end{align*}
Now it remains to show that the Hessian matrix $ \partial_{\bar y}^2 \Phi_u(0,0,0) $  is positive definite.  This is easy to verify.  Observe $\nabla_{\bar  y} A(0,0,0)=\nabla_{\bar y} B(0,0,0)=0$. Thus we get 
\[ \partial_{\bar  y}^2  \Phi_u(0,0,0) = A\partial_{\bar y}^2 B(0,0,0) = \frac{1}{1-u^2} I_{d-1} . \] 
This verifies that  the surface $\bar y\to  \nabla_{x} \arccos\phi_u(\bar x,  v, \bar y)$ has positive definite fundamental form. Hence we get the bound \eqref{nearequator1} and we see that the constant $C$ in \eqref{nearequator1} can be taken to be uniform because  the Hessian matrix $\partial_{\bar  y}^2  \Phi_u(0,0,0)$ can be controlled uniformly for $u\in ( 100r^2-1, 1-100 r^2)$.

\subsection{Near the north pole} We now consider the case that  $f$ satisfies \eqref{north}.  For the rest of this section $f$ is assumed to satisfy \eqref{north}. 

We start with observing 
\[  |\mathbf Z_n (\cos\theta)|  \lesssim n^{d-1}, \] 
which is easy to see using \eqref{zkernel}  and \eqref{asymp}. This gives the sharp $L^1$ to $L^\infty$ bound for $H_n^d$.  Using this,  the contribution from  the part of kernel $\theta\lesssim N^{-1}$  is  easy to  handle.  In fact, let $\psi$ be a smooth function supported in  $[-2,2]$ such that $\psi=1$ on  $[-1,1]$. Then we have 
\[  \Big| \int_{\mathbb S^d}  \mathbf Z_n (\zeta\cdot\xi) \psi\Big(\frac{n^2(1-\xi\cdot\zeta)}{100}\Big)    f(\xi) d\sigma(\xi)\Big | \lesssim n^{d-1}\int_{|\xi-\zeta|\le 20n^{-1}}      |f(\xi) | d\sigma(\xi).\] 
From a simple computation (Young's convolution inequality)  we get 
\begin{equation} \label{nearzero} 
\Big \| n^{d-1}\int_{|\xi-\zeta|\le 20n^{-1}}       |f(\xi) | d\sigma(\xi) \Big\|_{L^q(\mathcal C(e_{d+1}, r ))}\lesssim n^{-1+ d(\frac1p-\frac1q)} \|f\sl p. 
\end{equation} 
Since the bound $n^{-1+ d(\frac1p-\frac1q)}$   is acceptable, now we only need to consider the case $ 10 n^{-1}\le \theta\le 50r$. 

For the range $ 10 n^{-1}\le \theta\le 50r$,  we use the asymptotic expansion \eqref{asymp}.  Hence, as before,  it is enough to control the leading term $\mathcal A_\nu (\theta) A_0(\theta) N^{-\nu} J_{\nu}(N\theta)$ in \eqref{asymp2} as the lower order terms can be handled in the same way and these give smaller bounds.  Combining  this, the asymptotic expansion for the Bessel function and  \eqref{zkernel},   it suffices  to consider  zonal  convolution with   
\[   \mathcal Z(\cos\theta)=  N^{\frac{d-1} 2} \mathcal A_\pm(\theta)  e^{\pm iN\theta}, \]
where  $ \mathcal A_\pm $  is smooth and supported in  $(10n^{-1}, 50r)$ with  $  \frac{d^k}{d\theta^k} \mathcal A_\pm(\theta)=O(\theta^{-\frac{d-1}{2} -k}).$ Thus, using the typical dyadic partition of unity,   we may break $\mathcal Z(\cos\theta)$ dyadically such that 
\[    \mathcal Z(\cos\theta)=\sum_{j:  n^{-1}\le  2^{-j }\le  50r }    \mathcal Z_j(\cos\theta):=     N^{\frac{d-1}{2}} \sum_{j:  n^{-1}\le  2^{-j }\le  50r }   2^{\frac{d-1} 2 j} \psi_j(\theta) e^{\pm iN\theta},      \] 
where   $\psi_j(\theta)$ is supported in $[ 2^{-j-1}, 2^{-j+1}]$ and  $\|\psi_j( 2^{-j} \cdot)\|_{C^l([2^{-1}, 2^2])}$ is uniformly bounded for any $l$. Let us set    
\[   T_j f(\zeta)= \int   \mathcal Z_j(\zeta\cdot\xi) f(\xi) d\sigma(\xi). \]      
We claim that, for $p$, $q$ satisfying \eqref{ca-so} (and $f$ satisfying \eqref{north}) and $q>\frac{2(d+2)}d$, 
\begin{equation}\label{tj}
\|T_j f\|_{L^q(\mathcal C(e_{d+1}, r ))}  \lesssim  N^{\frac{d-1}{2}-\frac dq} 2^{(\frac dp -\frac{d+1}2)j} \|f\sl p. 
\end{equation}
 This gives, for  $p$, $q$ satisfying \eqref{ca-so}, 
\[ \Big\| \sum_{j:  n^{-1}\le  2^{-j }\le  50r } T_jf  \Big\|_{L^q(\mathcal C(e_{d+1}, r ))} \lesssim  \begin{cases}     N^{\frac{d-1}2-\frac dq}  \|f\sl p,  &    \frac{2d}{d+1} <  p,  \, q>\frac{2(d+2)}d , \\   
 N^{d(\frac 1p-\frac1q)-1}  \|f\sl p,  & 1\le p<   \frac{2d}{d+1}.
 \end{cases}
\]

For the critical case $p=\frac{2d}{d+1}$, we use a simple summation  lemma which was implicit in \cite{B}. A statement for a general multilinear setting can also be found in \cite{CKLS}.
\begin{lem}\label{interpol}  
Let $\e_0, \, \e_1 > 0$, and let $\{\mathcal T_l:l\in\Z\}$ be a sequence of linear operators satisfying
\[
\|\mathcal  T_l f\|_{q_0} \le M_0 2^{\e_0{l}}\|f\|_{p_0}, \quad \|\mathcal  T_l f\|_{q_1} \le M_ 1 2^{-\e_1{l}}\|f\|_{p_1}
\]
for some $1\le p_0,\, p_1,\, q_0,\, q_1\le \infty$. Then $\| \sum_l\mathcal  T_l f\|_{q,\infty} \le C M_0^\theta M_1^{1-\theta} \|f\|_{p,1},$ where $\theta =\e_1/(\e_0+\e_1)$, $1/q=\theta/q_0+(1-\theta)/q_1$, $1/p = \theta / p_0 +(1-\theta) / p_1$.  Additionally, if $q_0=q_1$, then  $\|\sum_l\mathcal T_l f\|_{q} \le C M_0^\theta M_1^{1-\theta} \|f\|_{p,1}.$
\end{lem}

Hence, by Lemma  \ref{interpol} and \eqref{tj} we have 
\begin{gather*}  
\Big\| \sum_{j:  n^{-1}\le  2^{-j }\le  50r } T_jf  \Big\|_{L^q(\mathcal C(e_{d+1}, r ))} \lesssim    N^{\frac{d-1}2-\frac dq}  \|f\|_{L^{\frac{2d}{d+1} , 1}(\mathbb S^d)} ,   \quad  q> \frac{2d(d+1)} {(d-1)^2}, \\   
\Big \| \sum_{j:  n^{-1}\le  2^{-j }\le  {50r} }  T_jf  \Big\|_{L^{q,\infty}(\mathcal C(e_{d+1}, r ))}  \lesssim   N^{\frac{d-1}2-\frac dq}  \|f\|_{L^{\frac{2d}{d+1} , 1}(\mathbb S^d)} ,  \quad q= \frac{2d(d+1)} {(d-1)^2} .  
\end{gather*} 

Now we can repeat the same argument with  $ \mathcal A_\nu (\theta) A_l(\theta) N^{-\nu-l} J_{\nu+l}(N\theta)$, $1\le l\le m-1,$ and the contributions from these terms are controlled by smaller norms.  Hence combining this with \eqref{nearzero} we conclude that for  $p$, $q$ satisfying \eqref{ca-so}, 
\begin{equation}\label{hh1} 
\| H_n^d f \|_{L^{q}(\mathcal C(e_{d+1}, r ))}   \lesssim \begin{cases} 
N^{\frac{d-1}2-\frac dq} \|f\sl p \,\, &\frac{2d}{d+1} <  p,  \, q>\frac{2(d+2)}d, \\  
N^{d(\frac 1p-\frac1q)-1} \|f\sl p \,\, &1\le p< \frac{2d}{d+1} , 
\end{cases} 
\end{equation} 
and 
\begin{equation}\label{hh2} 
\begin{aligned}  
\| H_n^d f \|_{L^{q}(\mathcal C(e_{d+1}, r ))}  &\lesssim    N^{\frac{d-1}2-\frac dq}   \|f\|_{L^{\frac{2d}{d+1} , 1}(\mathbb S^d)}  ,   & q> \frac{2d(d+1)} {(d-1)^2}, \\   
\| H_n^d f  \|_{L^{q,\infty}(\mathcal C(e_{d+1}, r ))}     &  \lesssim   N^{\frac{d-1}2-\frac dq}   \|f\|_{L^{\frac{2d}{d+1} , 1}(\mathbb S^d)} ,  & q= \frac{2d(d+1)} {(d-1)^2} .  
\end{aligned} 
\end{equation}
 
\begin{proof}[Proof of Theorem \ref{k-proj}] Finally, combining these  estimates  \eqref{hh1} and \eqref{hh2}, with  \eqref{nands} and using rotational symmetry, we see that the same bounds are true for $H_n^d$ by replacing  $\mathcal C(e_{d+1}, r )$ with $\mathbb S^d$  without restriction on $f$ for $p$, $q$ satisfying \eqref{ca-so}. Since the estimates for $(p,q)=(2,2)$ and $(p,q)=(1,\infty)$ are trivially true,  interpolation and duality  yield all the bounds in Theorem \ref{k-proj}. 
\end{proof}  
 
\begin{proof}[Proof of \eqref{tj}]  We now use the  parametrization   $ \kappa (x)=(x, \sqrt{1-|x|^2})$ of the sphere near $e_{d+1}$ for both $\zeta$ and $\xi$.   %
Let us set 
\[  \phi(x,y)= x\cdot y+ \sqrt{(1-|x|^2)(1-|y|^2)}.\] 
Then, we may write
\[  T_jf (x)=   N^{\frac{d-1}{2}}2^{\frac{d-1} 2 j}  \int    \psi_j(\arccos   \phi(x,y) ) e^{iN\arccos   \phi(x,y)}   f(\kappa(y) )w(y)  dy, \]
where $w(y)=\sqrt{1+ \frac{ |y|^2}{1-|y|^2}}$.    By rescaling  $(x,y)\to 2^{-j} (x,y)$ and  recalling the definition of  $\Phi_\epsilon$  (in Lemma \ref{converge}) we see that 
\[ T_jf(2^{-j} x)= N^\frac{d-1}2   2^{-\frac{d+1}2 j} \int    \psi_j(2^{-j}  \Phi_{2^{-j}}(x,y) )  e^{i2^{-j}N\Phi_{2^{-j}}(x,y) }f(\kappa(2^{-j}y) ) w(2^{-j}y)  dy. \] 
By Lemma  \ref{converge} we see that  the functions $\psi_j(2^{-j}  \Phi_{2^{-j}}(x,y) )$ are nonzero only if $|x-y|\in [2^{-2}, 2^{2}]$ and are uniformly bounded  in $C^M(\R^d\times \R^d)$ for any $M$ provided that  $j$ is large enough so that $2^{-j}\le 50 r\le \epsilon_0/2$.  For $n^{-1} \le 2^{-j}\le 50 r$,  we set 
\[ \widetilde T_j g (x)=   \int     \psi_j(2^{-j}  \Phi_{2^{-j}}(x,y) )  e^{i2^{-j}N\Phi_{2^{-j}}(x,y) } g(y)   dy.\] 
Discarding some harmless factors which arise from the parametrization, we  see  \eqref{tj} is equivalent to  the estimate
\begin{equation}
\|  \widetilde T_j f\|_q\lesssim  (2^{-j}N)^{-\frac dq}  \|f\|_p
\end{equation} 
for $p$, $q$ satisfying \eqref{ca-so}.  This follows from Proposition \ref{osc} and our choice of $r$. 
\end{proof}

\section{Lower bound for $\hpqnorm$: Proof of Theorem \ref{nec}} \label{necpr}
In this section we prove  the lower bound for $\hpqnorm$ in Theorem \ref{nec} by  testing the equality $\| H^d_n f \sl{q}\le C\|f \sl{p}$ with various input functions $f$. 
For $p$, $q$ in a certain range, this can be done using spherical harmonic functions of  degree $n$, but  for general $p$, $q$  we need to analyze  the  integral operator  $H_n^d$ directly using \eqref{zkernel} and \eqref{asymp}. 

In order to show \eqref{lower}, by duality it is enough to show  that 
\begin{align}
& \hpqnorm \gtrsim  n^{\frac{d-1}2 -\frac dq},    \label{bd3} \\   
&  \hpqnorm \gtrsim n^{\frac{d-1}2(\frac1p-\frac1q)},  \label{bd1} \\
&  \hpqnorm \gtrsim n^{d(\frac1p-\frac1q) -1} . \label{bd2}
\end{align}

\subsubsection*{Proof of  \eqref{bd3}}  On the limited range $p<\frac{2d}{d-1} <q$, \eqref{bd3} can be shown with the following estimate from  Szeg\"o  {\cite[p. 391]{Sz1}}. 
 
\begin{lem} 
For $\alpha, \beta, \mu>-1$, and   $p>0$, 
\[ \int_0^1(1-t)^\mu |P^{(\alpha, \beta)}_n (t)|^p dt  \sim  
\begin{cases}   n^{\alpha p-2\mu-2}, & 2\mu<  \frac{2\alpha+1}2 p-2, \\ 
n^{-\frac p2} \log n,  & 2\mu=\frac{2\alpha+1}2 p-2, \\ 
n^{-\frac p2}, & 2\mu>  \frac{2\alpha+1}2 p-2. 
\end{cases} 
\] 
\end{lem} 
 
In particular,  taking $\alpha=\beta=\mu=\frac{d-2}2$, we see that, for $e\in \mathbb S^d$, 
\[ \| P_n^{(\frac{d-2}2, \frac{d-2}2)} (\xi\cdot e) \sl{p }\sim 
\begin{cases}   n^{\frac{d-2}2-\frac dp}, & p> \frac{2d}{d-1} , \\ 
n^{-\frac 12} (\log n)^\frac1p, & p=\frac{2d}{d-1}, \\ 
n^{-\frac 12}, & p< \frac{2d}{d-1}. 
\end{cases}
\]
Testing the inequlity with $f(\xi)=P_n^{(\frac{d-2}2, \frac{d-2}2)} (\xi\cdot e)$ which is in $\mathcal H_n^d$ gives that,  for  $p<\frac{2d}{d-1} <q$,  $ \|H_n^d\|_{p, q}\gtrsim  n^{\frac{d-1}2-\frac dq}.$ 

To get the condition on the full range, we need to consider  the kernel of  the operator $H_n^d$. For this, using \eqref{asymp} and the asymptotic expansion of the Bessel function it is easy  to see that, for $N\gg 1$, 
\[ P_n^{(\frac{d-2}2, \frac{d-2}2)}(\cos\theta)= \frac{2^\frac{d-1}{2}\Gamma(n+\frac d2)}{\sqrt\pi n!} N^{-\frac{d-1}2} (\sin\theta)^{-\frac{d-1}2} \Big( \cos\Big( N\theta-\frac{(d-1)\pi}4 \Big)+ O(N^{-1})\Big)\] 
provided that $\theta$ is away from zero, say, $\theta\in [\pi/6,  \pi/2]$.  Combining this with \eqref{zkernel}, we have that, if  $0 \le \zeta \cdot\xi \le \sqrt3/2 $, 
\begin{equation}\label{zonal} 
\mathbf Z_n(\zeta\cdot\xi) =  C(N) G(\arccos \zeta\cdot\xi) \Big( \cos\Big( N\arccos \zeta\cdot\xi -\frac{(d-1)\pi}4\Big)+ O(N^{-1})\Big)
\end{equation}
with $C(N)\sim N^\frac{d-1}2$ and $G\sim 1$.  Hence, if $|\zeta-e_{d+1}|\le c N^{-1}$ for a small constant $c>0$  and  $1/10\le \zeta \cdot\xi \le \sqrt3 /2$,  then
\begin{equation}\label{zon2}
\mathbf Z_n(\zeta\cdot\xi) =C(N) G(\arccos \zeta\cdot\xi)  \Big( \cos\Big( N\arccos e_{d+1}\cdot\xi -\frac{(d-1)\pi}4 +O(c) \Big)+ O(N^{-1})\Big).
\end{equation} 
Let us set 
\[\mathlarger\Sigma =\bigcup_{k\in \mathbb Z: \frac{N}{12}\le k\le \frac{N}{6} }  \Big\{\xi\in \mathbb S^d:  N \arccos (e_{d+1}\cdot\xi) -\frac{(d-1)\pi}4\in  [2\pi k ,\frac \pi 4 +2\pi k   ]  \Big   \} \]   
and set  $ f(\xi)= \mathbf 1_{\Sigma}(\xi). $  Then it is easy to see that  $\|f\|_p\sim 1$  for any large $N$ and $p\in [1,\infty]$. By \eqref{zon2}, we see that   $\mathbf Z_n(\zeta\cdot\xi)\sim N^{\frac{n-1}2}$ if $|\zeta-e_{d+1}|\le c N^{-1}$ and $\xi \in \Sigma$.  Hence, for  $|\zeta-e_{d+1}|\le c N^{-1}$, 
\[ H_n^df(\zeta)=\int_{\mathbb S^d}  \mathbf Z_n(\zeta\cdot\xi)   f(\xi) d\sigma(\xi) \sim    N^\frac{d-1}2.\] 
This implies that   $\| H_n^df\sl q\gtrsim N^{\frac{d-1}2 -\frac dq}$   while   $\|f\sl p\sim1$. So, we get  $ \| H_n^d\|_{p,q}\gtrsim N^{\frac{d-1}2 -\frac dq}.$

\subsubsection*{Proof of  \eqref{bd1}}  This can be shown by making use of Gaussian beam. 
Let us consider  
\[  h_n(\eta)= (\eta_1+ i\eta_2)^n, \quad \eta=(\eta_1, \eta_2, \eta' ) \in \mathbb S^d, \quad  \eta_1,\, \eta_2 \in \mathbb R .\] 
Then, clearly, $H_n^dh_n=h_n$ and it is well known that $ \|h_n\|_{L^p(\mathbb S^d)}  \sim   n^{-\frac{d-1}{2p}}.$  In fact, $|h_n(\eta)|=  e^{n\ln \sqrt{ \eta_1^2+\eta_2^2}  } =e^{\frac n2\ln ( 1-|\eta'|^2)} \sim e^{-\frac{n}{2} {\text{dist} (\gamma, \eta)^2} }$.  Here $\gamma$ is the great circle which  is contained in $\eta'=0$. Hence it follows that $\|H_n^d\|_{p,q}\gtrsim n^{\frac{d-1}2(\frac1p-\frac1q)}$.

\subsubsection*{Another proof of  \eqref{bd1}}  It is also possible to show this directly without using special spherical harmonics. We make use of \eqref{zonal} and the parametrizations of $\mathbb S^d$ near $e_2$ and $e_1$, respectively, 
\[ \zeta=\Big(x_1, \sqrt{1-x_1^2-|\tilde x|^2},  \tilde x\Big),\,\,   \xi=\Big (\sqrt{1-y_2^2-|\tilde y|^2}, y_2,  \tilde y\Big),\,\, |x_1|,\, |y_2|\le c, \,\,  |\tilde x|,\, |\tilde y|\le cN^{-\frac12} \]
with a small enough $c>0$. Then, we have 
\[ \zeta\cdot \xi= x_1\sqrt{1-y_2^2}+ y_2\sqrt{1-x_1^2} + O(c^2N^{-1}).\]  
Putting $x_1=\cos\theta$ and $y_2=\sin\theta'$ with $|\theta-\pi/2|\le 2c$ and $|\theta'|\le2c$,  we have from \eqref{zonal}
\begin{equation}\label{decomp}
\mathbf Z_n(\zeta\cdot\xi)=C(N)  G(\arccos \zeta\cdot\xi)  \Big(\cos\Big( N(\theta-\theta')-\frac{(d-1)\pi}4+O(c^2) \Big)+ O(N^{-1})\Big)
\end{equation} 
because $\zeta\cdot\xi=\cos(\theta-\theta')+O(c^2N^{-1})$ and $\arccos(s+t)=\arccos s+ O(t)$. With a small enough $c>0$, set 
\[ {\mathlarger\Sigma}=\bigcup_{l\in \mathbb Z : |l|\le \frac{cN}{4\pi}}\Big\{ (\theta', \tilde y):  \theta' \in [-\frac cN, \frac cN] + \frac{2\pi l}N, \, |\tilde y|\le cN^{-\frac12}\Big\}.  \] 
Let us define a function $f$ on the sphere by setting $f(\sqrt{\cos^2\theta'-|\tilde y|^2}, \sin\theta', \tilde y)= \mathbf 1_\Sigma(\theta', \tilde y)$, and integrate $|H_n^df(\cos\theta, \sqrt{\sin^2\theta-|\tilde x|^2}, \tilde x)|^q$ over $(\theta, \tilde x)\in \Sigma+\frac{\pi}2$. If we choose $N\gg 1$ so that $2N-d+1\in 4\N$ and ${c>0}$ small enough, then in \eqref{decomp} we see that
\[\Big| \cos\Big( N(\theta-\theta')-\frac{(d-1)\pi}4+O(c^2) \Big)+ O(N^{-1})\Big|\gtrsim 1. \]
So,  $\|H_n^d f\sl q\gtrsim N^{-\frac{d-1}{2q}}$ while $\|f\sl p\sim N^{-\frac{d-1}{2p}}$. Hence, the desired  \eqref{bd1} follows.

\subsubsection*{Proof of \eqref{bd2}}   Using the fact that   $\widetilde J_k(r)=r^{-k}J_k(r)$  is a well defined analytic function and $\widetilde J_k(0)>0$, $k\ge 0$, it is easy to see  from \eqref{asymp} that $P_n^{(\frac{d-2}2, \frac {d-2}2)}(\cos\theta)\gtrsim n^{\frac{d-2}2}$ if $0\le \theta <cN^{-1}$ with $c$ small enough.   Hence, by \eqref{zkernel} we have that,  for $0<\theta\le cN^{-1}$, 
\[\mathbf Z_n(\theta)\gtrsim N^{d-1} .\]
Hence if we consider the function $f=\mathbf 1_{|\xi-e_{d+1}|\ll N^{-1}}$, then  $H_n^df(\zeta)\gtrsim N^{-1}$ if $|\zeta- e_{d+1}|\ll N^{-1}$.  This gives $ \| H_n^df\sl q/ \| f\sl p\gtrsim   N^{d(\frac1p-\frac1q)-1}$ and \eqref{bd2}.

We now show \eqref{intervalrs} when  $\pqpair\in [S,R]\cup  [S',R']$.  For this we need the following which  is an extension of  a lemma  in Sogge \cite{So}.
\begin{lem} \label{limit} 
Let $p, q\in [1,\infty]$.  Suppose that   there exists a constant $B$ such that 
\begin{equation}\label{assume} 
\hpqnorm \le  B\, n^{d(\frac1p-\frac1q)-1}.
\end{equation} 
Then  we have, for $f\in \mathcal S(\mathbb R^d)$, 
\begin{equation}\label{rest-ext}
\Big\| \int_{\mathbb S^{d-1}}  e^{ i x \cdot \eta }  \widehat  f(\eta) \, d\sigma(\eta) \Big\|_{L^q(\mathbb R^d)} \lesssim  B \|f\|_{L^p(\mathbb R^d)}.
\end{equation}
\end{lem}

The Bochner-Riesz operator $R^{\alpha}$ of  order $\alpha$ is the multiplier operator defined by
\[ \widehat {R^{\alpha} f} (\xi) =\frac{1}{\Gamma(\alpha+1)} (1-|\x|^2)_+^\alpha \widehat f(\xi), \quad \xi\in \mathbb R^d. \]
The definition  $R^{\alpha} f$  can be extended   to $\alpha\le -1$ by analytic continuation from the above formula. $L^p$--$L^q$ boundedness for  the Bochner-Riesz operator of negative order has been studied {by} some authors \cite{Boj, So, Carbery-Soria, Bak, CKLS}, and  it was shown by B\"orjeson \cite{Boj} (also see \cite{Bak}) that $R^{\alpha}$, $\alpha<0$,  is bounded from $L^p$ to $L^q$ only if 
\[  \frac1p>\frac{d-1-2\alpha}{2d},\quad   \frac1q<\frac{d+1+2\alpha}{2d} , \quad  \frac1p-\frac1q\ge \frac{-2\alpha}{d+1} . \] 
The problem of  $L^p$--$L^q$ boundedness of $R^{\alpha}$ is completely settled in $\mathbb R^2$ and in higher dimensions the sharp boundedness is established for $\alpha<-\frac{d^2-d-2}{2(d^2+d-2)}$ (see \cite{Bak, CKLS}).

\subsubsection*{Proof of \eqref{intervalrs} for $\pqpair\in [S,R]\cup  [S',R']$}  Since $\frac{1}{\Gamma(\alpha+1)} t_+^\alpha$ equals  the delta distribution when $\alpha=-1$,  it follows that 
$R^{-1} f(x)= \frac12 \int_{\mathbb S^{d-1}}  e^{ i x \cdot \eta }  \widehat  f(\eta) \, d\sigma(\eta)$.  Hence,  from the above we see that \eqref{rest-ext} is possible only for $p$, $q$ satisfying $\frac1p >\frac{d+1}{2d},$ $ \frac1q< \frac{d-1}{2d},$ $\frac1p-\frac1q > \frac{2}{d+1}.$ This is equivalent to $\pqpair\in \mathfrak T_2\setminus ([S,R]\cup[S',R'])$. Now let   $\pqpair\in  ([S,R]\cup[S',R'])$ and suppose  $ n^{-\gamma(p,q)} \hpqnorm\le B$ for some constant $B$.    Then, by Lemma   \ref{limit} the estimate \eqref{rest-ext} should also be true. This contradicts aforementioned  necessary condition.  Hence, there is no $B$ such that  \eqref{assume} holds when $\pqpair\in [S,R] \cup [S', R']$. This gives \eqref{intervalrs}. \qed 

\begin{proof} [Proof of Lemma \ref{limit}] 
Let  us define a map $\mu_n: \mathbb R^d\to   n \mathbb S^d\setminus \{ne_{d+1}\}$  by
\[
\mu_n(x)=\Big( \frac{4n^2 x}{|x|^2+4n^2},  \frac{n |x|^2-4n^3}{|x|^2+4n^2}\Big).
\] 
If $S_n$ is  the stereographic projection of $n \mathbb S^d\setminus \{ne_{d+1}\}$ to $\mathbb R^d\times\{-n\}$, then 
$S_n(\mu_n(x) )=(x,-n)$. 

Let $d\sigma_n$ and $dx$ denote the surface measure on $n\mathbb S^d$ and the Lebesgue measure on $\R^d$, respectively. By using rotational symmetry and  a computation it is easy to see that $\sqrt{\det((D\mu_n)^{t}D\mu_n)}=(\frac{|x|^2}{4n^2}+1)^{-d}$. Thus we have   
\begin{equation}\label{jacob}
 {d\sigma_n(x)}= \Big(\frac{|x|^2}{4n^2}+1\Big)^{-d} dx.
 \end{equation}

Let $f, g\in \mathcal S(\mathbb R^d)$ and consider the integral 
\begin{align*}
I_n &=   n^{d+1} \int_{\mathbb S^d}  \int_{\mathbb S^d}   \bZ_n( {\xi\cdot\eta}  )   f(\mu^{-1}_n (n \xi))   g(\mu^{-1}_n (n\eta)) \,d\sigma(\xi)\, d\sigma(\eta)\\
&= n^{d+1} \int_{\mathbb S^d}   H_n^d  ( f(\mu^{-1}_n (n \cdot)))  (\eta)    g(\mu^{-1}_n (n\eta)) \, d\sigma(\eta) .
\end{align*}
From \eqref{jacob}  $\|f(\mu^{-1}_n (n \cdot) )\sl p \lesssim n^{-\frac dp} \|f\|_{L^{p}(\mathbb R^d)}$ and  $\|g(\mu^{-1}_n (n \cdot) )\sl {q'} \lesssim n^{-\frac d{q'}}  \|g\|_{L^{q'}(\mathbb R^d)}$. Hence, by \eqref{assume} it is easy to see 
\begin{equation} \label{pqbound} 
| I_n|\lesssim    n^{1- d(\frac 1p-\frac 1{q})} \|H^{d}_n\|_{p,q} \|f\|_{L^{p}(\mathbb R^d)} \|g\|_{L^{q'}(\mathbb R^d)}\,\lesssim B\|f\|_{L^{p}(\mathbb R^d)} \|g\|_{L^{q'}(\mathbb R^d)}.  
\end{equation}

On the other hand,  by changes of variables, 
\begin{align*} 
I_n&= \int_{n \mathbb S^d}  \int_{n \mathbb S^d}   {n^{1-d}} \bZ_n\Big(\frac{\xi\cdot\eta}{n^2}     \Big)   f(\mu^{-1}_n \xi) g(\mu^{-1}_n \eta) \, \,d\sigma_n(\xi)\, d\sigma_n(\eta) \\ 
&= \int_{ \mathbb R^d}  \int_{\mathbb R^d}   {n^{1-d}} \bZ_n\Big(\frac{\mu_n ( x) \cdot  \mu_n (y) }{n^2}     \Big)   f(x) g(y)\, \frac{d\sigma_n}{dx}\, \frac{d\sigma_n}{dy}    \,d x\, d y.  
\end{align*}

We recall  the identity known as  Mehler-Heine type  (see Szeg\"o \cite[p.192] {Sz1}) 
\[  \lim_{n\to \infty}     n^{1-d} \bZ_n\Big(\cos\Big\{ \frac r n+ o\Big(\frac1n\Big)\Big\}\Big)= c_d\, r^{-\frac{d-2}2} J_{\frac{d-2}2}(r)\] 
for some $c_d$.  This can be easily shown by using \eqref{zkernel} and \eqref{asymp}.    Also, note that 
\[\frac{\mu_n ( x) \cdot  \mu_n (y) }{n^2}= 1-\frac{|x-y|^2}{2 n^2}+O(\frac1{n^3})= \cos\Big\{ \frac{ |x-y| }n+ o\Big(\frac1n\Big)\Big\}.\] 
Hence, it follows that 
\[ \lim_{n\to \infty}{n^{1-d}} \bZ_n\Big(\frac{\mu_n ( x) \cdot  \mu_n (y) }{n^2}     \Big)=  c_d |x-y|^{-\frac{d-2}2} J_{\frac{d-2}2}(|x-y|)= \widetilde c_d    \int_{\mathbb S^{d-1}}  e^{i(x-y)\cdot \eta} d\sigma(\eta)  .\] 
Therefore, recalling $\lim_{n\to\infty} \frac{d\sigma_n}{dx}=1$ from \eqref{jacob},  we get 
\[ \lim_{n\to \infty}         I_n= \widetilde c_d \int_{\mathbb R^d} \int_{\mathbb S^{d-1}} e^{i y\cdot\eta}  \widehat f(\eta) \, d\sigma(\eta)\, g(y) \, dy .  \] 
Combining this with \eqref{pqbound} and duality  yield 
\[  \Big| \widetilde c_d \int_{\mathbb R^d} \int_{\mathbb S^{d-1}} e^{i y\cdot\eta}  \widehat f(\eta) \, d\sigma(\eta)\, g(y) \, dy  \Big| \lesssim  \|f\|_{L^{p}(\mathbb R^d)} \|g\|_{L^{q'}(\mathbb R^d)}\, . \]
Duality gives \eqref{rest-ext}. 
\end{proof}

\section{Application to Carleman estimate}  \label{appcar}

In this section we prove Theorem \ref{carleman}. For this it is sufficient to show the following since \eqref{proj} holds as long as $p,q$ satisfy \eqref{scaling} and $ \frac{2(d-1)d}{d^2+2d-4} < p < \frac{2(d-1)}{d}$  by Theorem \ref{k-proj}.

\begin{prop}\label{suf}Let $p,q \in (1,\infty)$ and satisfy \eqref{scaling}. Suppose  that the estimate 
\begin{equation}\label{proj}
\| H_n^{d-1} f\|_{L^q(\mathbb S^{d-1})}\lesssim n^{1-\frac2d} \| f\|_{L^p(\mathbb S^{d-1})}
\end{equation}
holds. Then, for the same $p,q$,  \eqref{weight} holds whenever $\dist(\tau, \mathbb Z+\frac d{q})\ge c$ for some $c>0$. 
\end{prop}

\newcommand{\ptr}{\partial_r}
\newcommand{\Norm}{ \Big\|} 

\begin{proof}[Proof of Proposition \ref{suf}] We follow  Jerison's idea in \cite{Je}. 

First, using the spherical coordinates $(r,\omega)\in \mathbb R_+\times \mathbb S^{d-1}$ and   the identity $\Delta=\partial_r^2+\frac{d-1}r \partial_r+ \frac{1}{r^2}\Delta_{\mathbb S^{d-1}},$
we  note that 
\[ |x|^{-\tau} \Delta\, |x|^{\tau}= \partial_r^2 +\frac{2\tau+d-1}r \partial_r +\frac{\tau^2+(d-2)\tau}{r^2}  + \frac{1}{r^2}\Delta_{\mathbb S^{d-1}}.\]
Making the change of variables $ r=e^{-t}$ gives $|x|^{-\tau} \Delta\, |x|^{\tau}= e^{2t} ( P(\partial_t)+ \Delta_{\mathbb S^{d-1}})$, where we set  \[P(t)= t^2 -(2\tau+d-2) t + \tau(\tau+d-2). \]  
Hence, we see that  \eqref{weight} is equivalent to the estimate 
\[
 \Big( \int_{-\infty}^\infty \int_{\mathbb S^{d-1}} |u(t, \omega)|^q d\omega\, e^{-d\,t}dt\Big)^\frac1q  \lesssim 
 \Big( \int_{-\infty}^\infty \int_{\mathbb S^{d-1}} \big|( P(\partial_t)+ \Delta_{\mathbb S^{d-1}})u(t,\omega)\big|^p d\omega\, e^{(2 p-d)t} dt\Big)^\frac1p 
\]
for $u\in C^\infty_c((\R\setminus\{0\}) \times \mathbb S^{d-1})$.

By replacing $ u\to e^{\frac dq t} u $ and then using the relations $\p_t^m (e^{\frac dq t} u) = e^{\frac dq t}(\p_t+\frac dq)^mu $ and $\frac dp-\frac dq=2$, it follows that the previous inequality is equivalent to
\begin{equation}\label{equiva}
 \Big( \int_{-\infty}^\infty \int_{\mathbb S^{d-1}} |u(t, \omega)|^q d\omega\, dt\Big)^\frac1q  \lesssim  
\Big( \int_{-\infty}^\infty \int_{\mathbb S^{d-1}} \big|( P\big(\partial_t+\frac dq\big)  + \Delta_{\mathbb S^{d-1}})u(t, \omega)\big|^p d\omega\,  dt\Big)^\frac1p. 
\end{equation}
Thus we are reduced to showing  
\begin{equation}\label{equiv}
\Big( \int_{-\infty}^\infty  \int_{\mathbb S^{d-1}} \big|\big[P\big(\partial_t+\frac dq\big) + \Delta_{\mathbb S^{d-1}}\big]^{-1}u(t, \omega)\big|^q d\omega\, dt\Big)^\frac1q   \lesssim   \Big( \int_{-\infty}^\infty \int_{\mathbb S^{d-1}} \big|u(t, \omega)\big|^p d\omega\,  dt\Big)^\frac1p. 
\end{equation}

The operator $[P\big(\partial_t+\frac dq\big)  + \Delta_{\mathbb S^{d-1}}]^{-1}$ can be expressed in terms of spherical harmonic projection operators. In fact, since   $\Delta_{\mathbb S^{d-1}}H_n^{d-1}=-n(n+d-2)H_n^{d-1}$  and  $ P\big(\partial_t+\frac dq\big)  -n(n+d-2) = (\partial_t+\frac dq-( \tau -n)) (\partial_t+\frac dq -(\tau +n+d-2)) $, it follows that  
\begin{equation}\label{spectral} 
 [P\big(\partial_t+\frac dq\big) + \Delta_{\mathbb S^{d-1}}]H_n^{d-1}= (\partial_t-\widetilde \tau+n) (\partial_t-\widetilde\tau -n-d+2) H_n^{d-1}, 
\end{equation}
where we set $\widetilde \tau=\tau - \frac dq $. Hence, taking Fourier transform in $t$,  we have that  
\begin{align} \label{second-express}
& [ P\big(\partial_t + \frac dq\big)+ \Delta_{\mathbb S^{d-1}}]^{-1} H_n^{d-1} u(t,\omega)
= \frac1{2\pi}   \int \frac{e^{i ts} \mathcal F_t (H_n^{d-1} u(\cdot, \omega))(s) \,  ds }{  (is-\widetilde \tau+n) (is-\widetilde\tau-n-d+2)  } \\
=  & \frac1{2\pi (2n+d-2)}   \int   e^{i ts} \Big( \frac1{ is-\widetilde\tau -n-d+2 } - \frac{1}{is-\widetilde \tau +n} \Big)\mathcal F_t (H_n^{d-1} u(\cdot, \omega))(s) \  ds \nonumber
 \end{align} 
if $2n\neq 2-d$, which is always true because $d\ge3$ and $n\ge 0$.  Here $\mathcal F_t$   denotes the Fourier transform with respect to  $t$ variable. 

\begin{lem}    Let $\alpha\neq0$. Then, for $g\in \mathcal S(\mathbb R)$, 
\[   \frac1{2\pi } \int  \frac{e^{its}} {is+\alpha} \mathcal F_t g(s) ds  = \begin{cases}  \int_{-\infty}^t  e^{-(t-s)\alpha} g( s) ds,  &  \alpha>0 \\  - \int_t^\infty e^{-(t-s)\alpha} g( s) ds,  &  \alpha<0 . \end{cases} \] 
\end{lem}
 
For the proof of this lemma it is enough to show that 
\[ \frac1{2\pi } \int  \frac{e^{its}} {is+\alpha} ds = \begin{cases} e^{-t\alpha}\mathbf 1_{(0,\infty)}( t)  ,  &  \alpha>0 \\   -e^{-t\alpha}\mathbf 1_{(-\infty, 0)}( t) ,  &  \alpha<0 \end{cases} \]
and this follows from an easy application of the residue theorem to the function $e^z/z$.
 
After applying spectral projection $u(t,w)=\sum_n H_n^{d-1}  u(t,\om)$, we see that the inverse of  $P\big(\partial_t+\frac dq\big)+ \Delta_{\mathbb S^{d-1}}$ is given as follows: 
\begin{equation}\label{inversex}
 [ P\big(\partial_t + \frac dq\big)+ \Delta_{\mathbb S^{d-1}}]^{-1} u(t,\omega) =\sum_{k=0} ^4  I_k u (t,\om), 
\end{equation}
where 
\begin{align*}
I_0 u(t,\om)&= \frac{1}{2\pi(d-2)} \int \frac{e^{its}\F_t( H_0^{d-1}u(s,\cdot))(s) }{(is-\ti\ta)(is-\ti\ta-d+2)}  ds, \\
I_1 u(t,\om)&= \sum_{ n > \widetilde \tau , \, n\neq 0}  \frac{-1}{ 2n+d-2} \int_{-\infty}^t e^{  -(n-\widetilde \tau) (t-s) } H_n^{d-1}u(s,\om) ds, \\
I_2 u(t,\om)&=\sum_{  0< n < \widetilde \tau }  \frac 1{2n+d-2} \int^{\infty}_t e^{  -(n- \widetilde \tau)  (t-s) } H_n^{d-1}u(s,\om) ds, \\
I_3 u(t,\om)&=\sum_{ n< -\widetilde \tau-d+2 , \, n\neq 0 }  \frac1{ 2n+d-2 } \int_{-\infty}^t  e^{ (\widetilde\tau+n+d-2)(t-s) } H_n^{d-1}u(s,\om) ds, \\
I_4 u(t,\om)&= \sum_{ n> -\widetilde \tau-d+2 ,\, n\neq 0 }  \frac {-1}{2n+d-2}  \int_t^\infty e^{ (\widetilde\tau+n+d-2) (t-s) } H_n^{d-1}u(s,\om) ds .
\end{align*} 
We see that $I_0u$ is trivially bounded from $L^p$ to $L^q$. Therefore, we are reduced to showing that, for $k=1,2,3,4$, 
\begin{equation}\label{inverse} 
\|I_k u\|_{L^q(dtd\omega)}\lesssim \|u\|_{L^p(dtd\omega)} 
\end{equation}  
with the implicit constant independent of $\tau$. We shall only prove \eqref{inverse} for $k=1,2$, since the others can be handled similarly.

{\it Proof of \eqref{inverse} for $k=1$.} We now use the spectral projection estimate \eqref{proj}, which is followed by Minkowski's inequality,  to get 
\begin{align*}
\|   I_1 u (t,\cdot)\|_{L^q(\mathbb S^{d-1})}  \lesssim  \int_{-\infty}^t \sum_{n>\widetilde \tau}  n^{-\frac 2d}   e^{-(n-\widetilde \tau) (t-s) }  \|u(s,\cdot)\|_{L^p(\mathbb S^{d-1})}  ds. 
\end{align*} 
Since $n-\wt \ta \ge c:=\dist(\tau, \mathbb Z+\frac d{q})>0$, whenever $s>0$, it is clear that $\sum_{n>\widetilde \tau}  n^{-\frac 2d}   e^{-(n-\widetilde \tau) s} $ is bounded by
\begin{align*}
\sum_{n>\widetilde \tau}  (n-\wt \ta)^{-\frac 2d}   e^{-(n-\widetilde \tau) s} &\le \sum_{j=0}^\infty (c+j)^{-\frac 2d} e^{-(c+j)s} \le c^{-\frac 2q} e^{-cs} + \int_0^\infty (c+u)^{-\frac 2d} e^{-(c+u)s}du \\
&\le e^{-cs}\Big( c^{-\frac 2q} + s^{\frac 2d-1}\Gamma (1-\frac 2d)\Big) \lesssim e^{-\tilde c s} s^{\frac 2d -1}
\end{align*} 
for some $\tilde c >0$. In fact, we can take any $\tilde c <c$. Therefore, by  Hardy-Littlewood-Sobolev inequality or Young's inequality, it follows that 
\begin{align*}
\|   I_1 u \|_{L^{\tilde q}(\R; L^q(\mathbb S^{d-1}))} & \lesssim \Big\|  \int_{-\infty}^t  (t-s)^{-\frac{d-2}{2}(\frac1p-\frac1q)} e^{-\tilde c(t-s)}  \|u(s,\cdot)\|_{L^p(dw)}  ds \Big\|_{L^{\tilde q}(\R, dt)} \\
& \lesssim \|u\|_{L^{\tilde p}(\R ; L^p(\mathbb S^{d-1}))} 
\end{align*} 
provided that  $1< \tilde p\le \tilde q< \infty$ and $   \frac1{\tilde p}-\frac1{\tilde q}\le \frac 2d $. In particular, taking   $\tilde p=p,\, \tilde q=q$, we get the desired estimate \eqref{inverse} for $k=1$. 
  
{\it Proof of \eqref{inverse} for $k=2$.} In this case we are assuming that $\ti\ta>1$, because otherwise the summation is empty. As before Minkowski's inequality and the spectral projection estimate \eqref{proj} give us
\[ \|   I_2 u (t,\cdot)\|_{L^q(\mathbb S^{d-1})}  \lesssim  \int^\infty_t \sum_{n=1}^{[\ti\tau]}  n^{-\frac 2d}   e^{(\widetilde \tau -n) (t-s)}  \|u(s,\cdot)\|_{L^p(\mathbb S^{d-1})}  ds, \]
where $[\ti\ta]$ is the largest integer less than $\ti\ta$. We note that $\ti\ta-[\ti\ta]\ge\dist(\ta, \Z+\frac dq)=c>0$.  To estimate the kernel let us fix $s<0$. Then an elementary computation shows that 
\begin{align}
\sum_{n=1}^{[\ti\ta]} n^{-\frac 2d} e^{(\ti\ta-n)s} &\le e^{-|s|\ti\ta} \Big( \int_1^{[\ti\ta]}(u-1)^{-\frac2d} e^{|s|u}du+e^{|s|[\ti\ta]}\Big) \nonumber \\
&= e^{-|s|\ti\ta} \Big(|s|^{\frac 2d-1}e^{|s|} \int_0^{|s|([\ti\ta]-1)}u^{-\frac2d} e^{u}du+e^{|s|[\ti\ta]}\Big). \label{ker}
\end{align}
If $1\le\ti\ta<2$, then \eqref{ker} is bounded by $e^{-|s|(\ti\ta-[\ti\ta])}\le e^{-c|s|}$, which is in $L^r((-\infty,0), ds)$ for $1\le r\le \infty$. If $\ti\ta\ge2$, then the integral in \eqref{ker} is bounded by
\begin{align*}
\int_0^{1} u^{-\frac 2d} e^u du  + {\mathbf 1}_{(\frac1{[\ti\ta]-1}, \infty)}(|s|) \int_{1}^{|s|([\ti\ta]-1)} u^{-\frac 2d} e^u du \lesssim 1+ {\mathbf 1}_{(1/[\ti\ta], \infty)}(|s|) e^{|s|([\ti\ta]-1)}
\end{align*}
with the implicit constant depends only on $c$ and $d$. Hence the quantity \eqref{ker} is bounded by
\[e^{-|s|(\ti\ta-1)}|s|^{\frac 2d-1} +  {\mathbf 1}_{(1/[\ti\ta], \infty)}(|s|) e^{-|s|(\ti\ta-[\ti\ta])} |s|^{\frac 2d-1} + e^{-|s|(\ti\ta-[\ti\ta])} \lesssim e^{-c|s|}|s|^{\frac 2d-1} + e^{-c|s|}.\]
Therefore as in the proof for $I_1$, we conclude that \eqref{inverse} is true for $k=2$. This completes the proof.
\end{proof}

\begin{rem}[Failure of \eqref{weight}]\label{fail}   We show \eqref{weight} does not hold if 
\begin{equation} 
\frac 1q< \frac{d-4}{2(d-1)}, \quad   \frac1p>   \frac{d+2}{2(d-1)}. 
\label{nec-carleman} 
\end{equation} 
We only need to consider $p,q$ satisfying \eqref{scaling} because \eqref{weight} implies the condition \eqref{scaling}. Let $n\gg 1$ and choose $\tau\notin\Z+\frac dq$ so that $\widetilde\tau=\tau-\frac dq=n+\frac12$. Let $h$ be a nontrivial smooth positive function supported in $[1/2,2]$ and $g$ be a spherical harmonic polynomial on $\mathbb S^{d-1}$ of degree $n$. Let us consider 
\[u(t,w)=   h(t)  g(w). \] 
By \eqref{spectral}  we have $\big( P(\partial_t + \frac dq\big)+ \Delta_{S^{d-1}}) u = O\big(n(|h(t)|+|h'(t)|+|h''(t)|) g(w)\big).$ Since \eqref{weight} and \eqref{equiva} are equivalent if \eqref{scaling} is satisfied, we apply \eqref{equiva} to the function  $u$. Hence, from integration in $t$  we have that 
\[   \| H_n^{d-1}\|_{p,q} \lesssim n.\] 
whenever \eqref{weight} holds.  However,   by Theorem \ref{nec}  such bound is possible only if \eqref{nec-carleman} is satisfied  when  $p,q$ satisfy \eqref{scaling}.  
\end{rem}

{\bibliographystyle{plain}

\end{document}